\documentclass[12pt,reqno]{amsart}

\usepackage{amssymb}
\usepackage{graphicx}
\usepackage{amsfonts}
\usepackage{dsfont}
\usepackage{tikz}
\usepackage{bm}
\usetikzlibrary{arrows}

\usepackage[ruled,vlined,algosection,norelsize]{algorithm2e}

\usepackage{mathptmx}

\numberwithin{equation}{section}

\def\Hreg{\operatorname{Hreg}}
\def\Hproj{\operatorname{Hprojdim}}
\def\decHreg{\operatorname{decHreg}}
\def\reg{\operatorname{reg}}
\def\Hdepth{\operatorname{Hdepth}}
\def\depth{\operatorname{depth}}

\def\rank{\operatorname{rank}}
\def\Ann{\operatorname{Ann}}

\def\Tor{\operatorname{Tor}}
\def\pos{\operatorname{p}}
\def\hdp{\operatorname{Hdepth}}
\def\dep{\operatorname{depth}}

\def\cP{{\mathcal P}}

\def\ZZ{{\mathbb Z}}
\def\NN{{\mathbb N}}

\def\mm{{\mathfrak m}}

\def\projdim{\operatorname{projdim}}

\newtheorem{lemma}[algocf]{Lemma}
\newtheorem{corollary}[algocf]{Corollary}
\newtheorem{theorem}[algocf]{Theorem}
\newtheorem{proposition}[algocf]{Proposition}

\theoremstyle{definition}
\newtheorem{definition}[algocf]{Definition}
\newtheorem{remark}[algocf]{Remark}

\newtheorem{example}[algocf]{Example}

\numberwithin{algocf}{section}

\title[Hilbert regularity of $\ZZ$-graded modules]{Hilbert regularity of $\ZZ$-graded modules \\ over polynomial rings}

\address{Universit\"at Osnabr\"uck, Institut f\"ur Mathematik, 49069 Osnabr\"uck, Germany}
\email{wbruns@uos.de}

\address{Universit\"at Osnabr\"uck, Institut f\"ur Mathematik, 49069 Osnabr\"uck, Germany}
\email{jmoyano@uos.de}

\author{Winfried Bruns}

\author{Julio Jos\'e Moyano-Fern\'andez}

\author{Jan Uliczka}

\address{Institut f\"ur Mathematik, Universit\"at Osnabr\"uck. Albrechtstra\ss e 28a,
D-49076 Osnabr\"uck, Germany}
\email{juliczka@uos.de}

\begin{document}

\thanks{The second author was partially supported by the Spanish Government Ministerio de Educaci\'on y Ciencia (MEC), grants MTM2007-64704 and MTM2012--36917--C03--03 in cooperation with the European Union in the framework of the founds ``FEDER''}

\begin{abstract}
Let $M$ be a finitely generated $\ZZ$-graded module over the standard graded polynomial ring $R=K[X_1, \ldots , X_n]$ with $K$ a field, and let $H_M(t)=Q_M(t)/(1-t)^d$ be the Hilbert series of $M$.
We introduce the Hilbert regularity of $M$ as the lowest possible value of the Castelnuovo-Mumford regularity for an
$R$-module with Hilbert series $H_M$.  
Our main result is an arithmetical description of this invariant which connects the Hilbert regularity of $M$ to the smallest $k$ such that the power series $Q_M(1-t)/(1-t)^k$ has no negative coefficients. Finally we give an algorithm for the computation of the Hilbert regularity and the Hilbert depth of an $R$-module.
\end{abstract}

\maketitle

\section{Introduction}

This note can be considered as part of a program that aims at
estimating numerical invariants of a graded module $M$ over a
polynomial ring $K[X_1,\dots,X_d]$ ($K$ is a field) in terms of
the Hilbert series $H_M(t)$. For the notions of commutative
algebra we refer the reader to Bruns and Herzog \cite{BH}.
Well-known examples of such estimates are the bound of Bigatti
\cite{Bi} and Hulett \cite{Hu} on the Betti numbers or the
bound of Elias, Robbiano and Valla \cite{ERV} on the number of
generators for ideals primary to $\mm=(X_1,\dots,X_d)$.

A more recent result is the upper bound on $\depth M$ (or,
equivalently, a lower bound on $\projdim M$) given by the third
author \cite{U}, namely the \emph{Hilbert depth} $\Hdepth M$.
It is defined as the maximum value of $\depth N$ for a module
$N$ with $H_M(t)=H_N(t)$. We must emphasize that we will always
consider the standard grading on $R$ under which all
indeterminates have degree $1$. As soon as this hypothesis is
dropped, matters become extremely difficult as witnessed by the
paper \cite{MU} of the second and third author.

The objective of this paper is to bound the Castelnuovo-Mumford
regularity $\reg M$ in terms of $H_M(t)$. Of course, the bound
is the lowest possible value of $\reg N$ for a module $N$ with
$H_M(t)=H_N(t)$, which we term \emph{Hilbert regularity}
$\Hreg M$.

Both Hilbert depth and Hilbert regularity can be computed in
terms of Hilbert decompositions introduced by Bruns,
Krattenthaler and Uliczka \cite{BKU} for arbitrary gradings; for a method computing Hilbert depth for $\ZZ^n$-graded modules see Ichim and the second author \cite{IM}.
The approach by Hilbert decompositions is related to Stanley depth and Stanley regularity; see
Herzog \cite{H} for a survey. Stanley regularity for quotients by 
monomial ideals was considered by Jahan \cite{S}. 
Also Herzog introduced Hilbert
regularity via decompositions.

Write $H_M(t)=Q(t)/(1-t)^d$ with $d=\dim M$ and $Q\in\ZZ[t]$
(we may certainly assume that $M$ is generated in degrees $\ge
0$). Then $\Hdepth M=d-m$ where $m$ is the smallest value of
all natural numbers $j$ such that $Q(t)/(1-t)^j$ is a positive
power series, i.e.\ a power series with nonnegative
coefficients \cite{U}. (Note that the Hilbert series
$Q(t)/(1-t)^d$ has nonnegative coefficients.) Hilbert
regularity cannot always be described in such a simple way, but
it is closely related to the smallest $k$ for which
$Q(1-t)/(1-t)^k$ is positive. See Theorems
\ref{theorem:partial} and \ref{theorem:complete}.

Our main tool for the analysis of Hilbert series are
presentations
$$
H(t)=\sum_{i=0}^{k-1} \frac{f_i t^i}{(1-t)^n} + \frac{c t^k}{(1-t)^{n}} + \sum_{j=0}^{d-n-1} \frac{g_jt^k}{(1-t)^{d-j}}
$$
that we call \emph{$(n,k)$-boundary presentations} since the
pairs of exponents $(u,v)$ occurring in the numerator and the
denominator of the terms $t^i/(1-t)^n$, $t^k/(1-\nobreak
t)^n$, and $t^k/{(1-t)^{d-j}}$ occupy the lower and the right
boundary of a rectangle in the $u$-$v$-plane whose right lower
corner is $(k,n)$.

Using the description of Hilbert regularity in terms of Hilbert
decompositions, one sees easily that $\Hreg M$ is the smallest
$k$ for which a $(0,k)$-boundary representation with
\emph{nonnegative} coefficients $f_i$, $c$, $g_j$ exists.
(Without the requirement of nonnegativity the smallest such $k$
is $\deg H_M(t)$.) The bridge to power series expansions of
$Q(1-t)/(1-t)^k$ is given by the fact that the coefficients
$g_j$ appear in such expansions.

The paper is structured as follows: we introduce Hilbert
regularity in Section \ref{sect:Hreg}, and discuss boundary
representations in Section \ref{sect:bound}. Hilbert regularity
is then determined in Section \ref{sect:ari}, whereas the last
section \ref{sect:alg} contains an algorithm that computes
Hilbert depth and Hilbert regularity simultaneously.

\section{Hilbert regularity}\label{sect:Hreg}

Let $K$ be a field and let $M$ be finitely generated graded
module over a positively graded $K$-algebra $R$. The
Castelnuovo-Mumford regularity of $M$ is given by
$$
\reg M=\max\{i+j:H_\mm^i(M)\neq0\}
$$
where $\mm$ is the maximal ideal of $R$ generated by the
elements of positive degree. If $R$ is a polynomial ring, then,
by a theorem of Eisenbud and Goto (see \cite{BH}, 4.3.1)
$$
\reg M=\max\{j-i:\Tor_i^R(K,M)_j\neq 0\}.
$$
where $K$ is naturally identified with $R/\mm$.

\begin{definition}
The  (plain) \emph{Hilbert regularity} of a finitely generated
graded $R$-module is
$$
\Hreg M=\min\{\reg N: H_N(t)=H_M(t)\}
$$
where $N$ ranges over the graded finitely generated
$R$-modules.
\end{definition}

Let $F_i$ be a graded free module over $K[X_1,\ldots ,X_i]$,
$i=1,\ldots, d$, considered as an $R$-module via the retraction
$R\to R[X_1,\dots,X_i]$ that sends $X_{i+1},\dots,X_d$ to $0$.
The module $F_0\oplus \dots\oplus F_d$ is called a
\emph{Hilbert decomposition} of $M$ if the Hilbert functions of
$M$ and $F_0\oplus \dots\oplus F_d$ coincide. This leads us to
the following definition:

\begin{definition}
The \emph{decomposition Hilbert regularity} of $M$ is
$$
\decHreg M=\min\{\reg N: H_N(t)=H_M(t)\}
$$
where now $N$ ranges over direct sums $F_0\oplus \dots\oplus
F_d$, i.e., over the Hilbert decompositions of $M$.
\end{definition}

It is in particular clear that $\decHreg M \geq \Hreg M$. As we
will see below, both numbers coincide in our setting of
standard graded polynomial rings. But both definitions make
sense in much more generality if one replaces the
$K[X_1,\ldots,X_i]$ by graded retracts of $K[X_1,\dots,X_d]$
(see \cite{BKU}). In the more general setting the equality is a
completely open problem, for regularity as well for depth. In
fact, proving equality for depth in the multigraded setting
would come close to proving the Stanley conjecture for depth
(see \cite{H}).

\begin{remark}
(a) The notion of Hilbert decomposition is the same as that in
\cite{BKU}, except that the $F_i$ are further decomposed into
cyclic modules there.

(b) Hilbert depth and Hilbert regularity are companions in the
following sense: the Hilbert depth determines the smallest
width of a Betti table admitting the given Hilbert series,
Hilbert regularity determines the smallest such possible
height. The Betti table is given in terms of the graded Betti
numbers $\beta_{i,j}=\dim_K\Tor_i^R (K,M)_j$ by
$$
\begin{matrix}
\beta_{0,0}&\beta_{1,1}&\dots&\beta_{p,p}\\
\vdots&\vdots&&\vdots\\
\beta_{0,r}&\beta_{1,r+1}&\dots&\beta_{p,r+p}
\end{matrix}
$$
where $p=\projdim M$ and $r=\reg M$.
\end{remark}

The decomposition Hilbert regularity can be described in terms
of \emph{positive representations} $\cP=(Q_d,\dots,Q_0)$ of the
Hilbert series:
$$
H_M(t)=\frac{Q_d(t)}{(1-t)^d}+\dots+\frac{Q_1(t)}{(1-t)^1}+Q_0(t),
$$
where each $Q_i$ is polynomial with nonnegative coefficients.
Such polynomials will be called \emph{nonnegative}. It is
well-known that  there is always a Hilbert decomposition of
$M$. This simple fact  will be proved (again) in Proposition
\ref{deco}.
\medskip

Let $F_0\oplus \dots\oplus F_d$ be a Hilbert decomposition of
$M$. Then we have
$$
H_{F_i}=Q(t)/(1-t)^i
$$
with a nonnegative polynomial $Q$, and we immediately get a
positive representation of the Hilbert series. Conversely,
given a positive representation of the Hilbert series, one
finds a direct sum $F_0\oplus\dots\oplus F_d$ by choosing $F_i$
as the free module over $R[X_1,\dots,X_i]$ that has $a_{ij}$
basis elements of degree $i$ where $Q_i=\sum_{j} a_{ij}t^j$.
\medskip

Moreover, $\reg F_i=\deg Q_i$, and therefore one has

\begin{proposition}\label{Pos}
$$
\decHreg M=\min_{\cP} \max_i \deg Q_i,\qquad \cP=(Q_d,\dots,Q_0),
$$
where $\cP$ ranges over the positive representations of
$H_M(t)$.
\end{proposition}

For Hilbert depth one can similarly give a ``plain'' definition
or a ``decomposition'' definition: The {\em Hilbert depth} of
$M$ is defined to be
\[
\Hdepth{M} := \max \left\{ r \in \NN ~\left |~ \begin{array}{l}
\mbox{there is a f.~g.~gr.~$R$-module $N$} \\ \mbox{with $H_N=H_M$ and $\dep N = r$.} \end{array}
 \right. \right\}.
 \]

The Hilbert depth of $M$ turns out to coincide with the
arithmetical invariant
 \[
 \pos{(M)} := \max \left\{ r \in \NN ~|~ (1-t)^r H_M(t)~\mbox{is nonnegative} \right\},
 \]
called the {\em positivity} of $M$, see Theorem 3.2 of
\cite{U}. The inequality  $\hdp M \leq \pos(M)$ follows from
general results on Hilbert series and regular sequences. The
converse can be deduced from the main result of \cite{U},
Theorem 2.1, which states the existence of a representation
\[
H_M(t) = \sum_{j=0}^{\dim M} \frac{Q_j(t)}{(1-t)^j}\qquad\mbox{with nonnegative}\quad Q_j \in \ZZ[t,t^{-1}].
\]

The decomposition version, or positivity, is close to Stanley
decompositions and Stanley depth. The same holds true for
Hilbert regularity, as we will see now; our proof will also
confirm the equivalence of the two notions of Hilbert depth.

\begin{proposition}\label{deco}
There exists a Hilbert decomposition of regularity equal to
$\reg M$ and depth equal to $\depth M$.
\end{proposition}

\begin{proof}
If $M$ is a free $R$-module, there is nothing to prove: $M$ is
already in Hilbert decomposition form.
\medskip

Now suppose that $M$ is not free. Let $m$ be the maximal degree
of a generator of $M$. Then $m \leq \reg M$, and we can choose
elements $v_1,\dots,v_n\in M$ of degree $\le m$ such that
$n=\rank M$ and $v_1,\dots,v_n$ are linearly independent. (This
is a well-known general position argument; we may have to pass
to an infinite field $K$, but that is no problem.) We set
$F_n=Rv_1+\dots+Rv_n$. For the sake of Hilbert series
computations we can replace $M$ by $F_n\oplus M/F_n$.
\medskip

Note that $\depth M/F_n=\depth M$ since $\depth M<\depth F_n$ by
assumption on $M$ and standard depth arguments. One has $\dim
M/F_n< n$ since $\rank M/F_n=0$ as an $R$-module.
\medskip

For the regularity we observe that  $M/F_n$ is generated
in degrees $\le m$ and ${\dim M/F_n< n}$. Since $F_n$ is free,
$\Tor_j^R(K,M/F_n)=\Tor^R_j(K,M)$ for $j\ge 2$, and therefore $1$
is the only critical homological degree for the regularity of
$M/F_n$. There is a homogeneous exact sequence
$$
\Tor_1^R(K,F_n)=0 \to \Tor_1^R(K,M)\to\Tor_1^R(K,M/F_n)\to\Tor_0^R(K,F_n)
$$
Except for $i\le m$, $\Tor_0^R(K,F_n)_i=0$, and
$\Tor_1^R(K,M)_i=\Tor_1^R(K,M/F_n)_i$. So the only critical
arithmetical degree is $m$. But we subtract $1$ from the
highest shift in homological degree $1$ in order to compute
regularity, and it does not matter for the inequality $\reg
M/F_n\le \reg M$ if $\Tor_1^R(K,M/F_n)_i\neq 0$ for some $i\le m$.
\medskip

On the other hand, $\reg M\le \max(\reg F_n,\reg M/F_n)$, and
altogether we conclude that $\reg M/F_n=\reg F_n$.
\medskip

Let $S=R/\Ann M$, and choose a degree $1$ Noether normalization
$R'$ in $S$. We view $M/F_n$ first as a module over $R'$. Then
$$
\reg_R M/F_n=\reg_S M/F_n=\reg_{R'} M/F_n
$$
since regularity does not change under finite graded
extensions. Now we can identify $R'$ with one of the algebras
$K[X_1,\dots,X_i]$ for some $i<n$. Hence we can proceed by
induction considering $M/F_n$.
\medskip

Eventually the procedure stops when the dimension of the
Noether normalization has  reached the depth of $M$ since the
quotient of $M$ reached then is free over the Noether
normalization, and we are in the case of a free module.
\end{proof}

\begin{remark}
The proof shows that regularity can be considered as a measure
for filtrations
$$
0= U_0\subset U_1\subset \dots \subset U_q=M
$$
in which $U_{i+1}/U_i$ is always a free module over some
polynomial subquotient of $R$: there exists such a filtration
in which each free module is generated in degree $\le \reg M$,
but there is no such filtration in which all base elements have
smaller degree. A similar statement holds for depth.
\end{remark}

\begin{corollary}
$$
\Hreg M=\decHreg M.
$$
\end{corollary}

In fact, if $N$ is a module whose regularity attains the
minimum, we can replace it by a Hilbert decomposition as in
Proposition \ref{deco}.
\medskip

A specific example: Let $M$ be the first syzygy module of the
maximal ideal in the polynomial ring $K[X_1,\dots,\allowbreak X_5]$. It has been
shown in \cite{BKU}, Theorem 3.5, that it has multigraded
Stanley depth $4$. It follows that the standard graded Hilbert
depth is also $4$, but this much easier to see: the Hilbert
series is
\begin{equation}\label{depthmax}
\frac{10t^2-10t^3+5t^4-t^5}{(1-t)^5}= \frac{10t^2}{(1-t)^4}+\frac{t^4}{(1-t)^4}+\frac{4t^4}{(1-t)^5}.
\end{equation}
So we can get away with the worst denominator $(1-t)^4$ for the
Hilbert depth.
\medskip

Let us look at he Hilbert regularity: the decomposition
\begin{equation}\label{regmin}
\frac{10t^2-10t^3+5t^4-t^5}{(1-t)^5}=\frac{4t^2}{(1-t)^5}+
\frac{3t^2}{(1-t)^4}+\frac{2t^2}{(1-t)^3}+\frac{t^2}{(1-t)^2}
\end{equation}
shows that $\Hreg M=2$. It cannot be smaller since $M$ has no
generators in degree $<2$. On the other hand, the decomposition
\eqref{regmin} is the only one with regularity $2$--and it
comes from a filtration as in the proof of Proposition
\ref{deco}. (In this example $\Hreg{M}$ could be determined more easily since $\Hreg{M}\geq 2$ and $\reg{M}=2$.) This shows that in general one cannot
simultaneously optimize depth and regularity.

More generally: if $M$ is a module with all generators in
degree $r$ and of regularity $r$, then $\Hreg M=\reg M$.

However, in general Hilbert regularity is smaller than
regularity: let $N$ be the sum of the modules in the Hilbert
decomposition \eqref{depthmax}, then $\Hreg N<\reg N$ as
\eqref{regmin} shows.

A simple lower bound:

\begin{proposition}\label{lower}
$$
\Hreg M \ge \deg H_M(t).
$$
\end{proposition}

In fact, for $j>\Hreg M$ the Hilbert polynomial and the
Hilbert function of $M$ coincide, and the smallest number $k$
such that the Hilbert polynomial and the Hilbert function
coincide in all degrees $j>k$ is $k=\deg H_M(t)$, the degree of $H_M$ as a rational
function; see \cite{BH}, 4.1.12.

\section{Boundary presentation}\label{sect:bound}

In this section we introduce the fundamental tool for our examination of the Hilbert regularity.

\begin{definition}\label{def:31}
Let $H(t)=Q(t)/(1-t)^d$. For integers $0\leq n\leq d$ and $k\geq 0$, an $(n,k)$-boundary presentation of $H$ is a decomposition of $H$ in the form
\begin{equation}\label{eq1}
H(t)=\sum_{i=0}^{k-1} \frac{f_i t^i}{(1-t)^n} + \frac{c t^k}{(1-t)^{n}} + \sum_{j=0}^{d-n-1} \frac{g_jt^k}{(1-t)^{d-j}} \mbox{~~with~~} f_i, c, g_j \in \ZZ.
\end{equation}
If $c=0$ the boundary presentation is called corner-free.
\end{definition}

Note that $Q(t)/(1-t)^d$ can be viewed as a $(d,\deg Q)$-boundary presentation of $H$. If $\deg Q \leq d$ there is also a $(d-\deg Q,0)$-boundary presentation:
let $Q(1-t)=\sum_{i} \tilde{q}_it^i$ then 
$$
H(t)=\frac{Q(t)}{(1-t)^d}=\frac{\sum_{i=0}^{\deg Q} \tilde{q}_i (1-t)^i}{(1-t)^d}=\sum_{i=0}^{\deg Q} \frac{\tilde{q}_i}{(1-t)^{d-i}}.
$$ 
In the sequel the polynomial $Q(1-t)$ will be needed several times, therefore we introduce the notation
$$
\tilde{Q}(t):=Q(1-t)
$$
for an arbitrary $Q \in \ZZ[t]$.

\begin{example}
Let $H(t)=\displaystyle \frac{1-2t+3t^3-t^4}{(1-t)^3}$. A $(1,3)$-boundary presentation of $H$ is given by
$$
H(t)=\frac{1}{1-t}+\frac{2t^2}{1-t}+\frac{2t^3}{(1-t)^2}+\frac{t^3}{(1-t)^3}.
$$

The term ``boundary presentation'' is motivated by visualisation of a decomposition of a Hilbert series: A decomposition
$$
\frac{Q(t)}{(1-t)^d}=\sum_{i=0}^{d} \sum_{j\geq 0} a_{ij} \frac{t^{j}}{(1-t)^{i}} 
$$
can be depicted as a square grid with the box at position $(i,j)$ labeled by $a_{ij}$. 
\medskip

\begin{center}
\begin{tikzpicture}[scale=0.8]
    \draw[] (0,0) grid [step=1cm](5,4);

    \node [below right] at (0.1,-0.2) {$\scriptstyle 0$};
    \node [below right] at (1.1,-0.2) {$\scriptstyle 1$};
    \node [below right] at (2.1,-0.2) {$\scriptstyle 2$};
    \node [below right] at (3.1,-0.2) {$\scriptstyle 3$};
    \node [below right] at (4.1,-0.2) {$\scriptstyle 4$};
 
    \node [below right] at (-0.9,0.8) {$\scriptstyle 0$};
    \node [below right] at (-0.9,1.8) {$\scriptstyle 1$};
    \node [below right] at (-0.9,2.8) {$\scriptstyle 2$};
    \node [below right] at (-0.85,3.8) {$\scriptstyle 3$};
         
     \node [below right] at (0.25,1.8) {$1$};
     \node [below right] at (1.25,1.8) {$0$};
     \node [below right] at (2.25,1.8) {$2$};
     \node [below right] at (3.25,1.8) {$0$};
     \node [below right] at (3.25,2.8) {$2$};
     \node [below right] at (3.25,3.8) {$1$};
\end{tikzpicture}
   \hspace{1cm}
   \begin{tikzpicture}[scale=0.8]
    \draw[] (0,0) grid [step=1cm](5,4);
    \node [below right] at (0.1,-0.2) {$\scriptstyle 0$};
    \node [below right] at (1.1,-0.2) {$\scriptstyle 1$};
    \node [below right] at (2.1,-0.2) {$\scriptstyle 2$};
    \node [below right] at (3.1,-0.2) {$\scriptstyle 3$};
    \node [below right] at (4.1,-0.2) {$\scriptstyle 4$};
           
    \node [below right] at (-0.9,0.8) {$\scriptstyle 0$};
    \node [below right] at (-0.9,1.8) {$\scriptstyle 1$};
    \node [below right] at (-0.9,2.8) {$\scriptstyle 2$};
    \node [below right] at (-0.85,3.8) {$\scriptstyle 3$};       
         
    \node [below right] at (0.25,3.8) {$1$};
    \node [below right] at (0,2.8) {$-1$};
    \node [below right] at (0.25,1.8) {$0$};
    \node [below right] at (0.25,0.8) {$1$};
\end{tikzpicture}
\end{center}
\begin{center}
{\small Two boundary presentations of $(1-2t+3t^2-t^3)/(1-t)^3$.}
\end{center}
\medskip

In case of an $(n,k)$-boundary presentation the nonzero labels in this grid form the bottom and the right edges of a rectangle with $d-n+1$ rows and $k+1$ columns. 
The coefficient in the ``corner" $(d-n,k)$ plays a dual role since it belongs to both edges, therefore it is denoted by an extra letter.
\end{example}

Next we deduce a description for the coefficients in a boundary presentation:

\begin{lemma} \label{lemma:32}
Let $H(t)=Q(t)/(1-t)^d$ be a series with $(n,k)$-boundary presentation (\ref{eq1}).
Moreover let
\[
\frac{Q(t)}{(1-t)^{d-n}}=\sum_{i=0}^{\infty} a_i t^i \ \ \ \ \mbox{and} \ \ \ \ \  \frac{\tilde{Q}(t)}{(1-t)^k}=\sum_{i=0}^{\infty} b_i t^i,
\] 
then 
\begin{align*}
f_i~&=~a_i \ \ \ \ \ \ \ \   \mbox{for} \ \ i=0, \ldots , k-1\\
&\\
c~&=~a_{k}-\sum_{i=0}^{d-n-1}b_i \ =\ b_{d-n}-\sum_{i=0}^{k-1}a_i\\
&\\
g_j&=~b_{j} \ \ \ \ \ \ \ \    \mbox{for} \ \ j=0, \ldots , d-n-1.
\end{align*}
\end{lemma}

\begin{proof}
Multiplication of (\ref{eq1}) by $(1-t)^n$ yields
$$
\frac{Q(t)}{(1-t)^{d-n}}=\sum_{i=0}^{k-1} f_i t^i +c t^k+ \sum_{j=0}^{d-n-1} \frac{g_jt^k}{(1-t)^{d-n-j}}.
$$
Hence the $f_i$ agree with the first $k$ coefficients of the power series $\sum_{i=0}^{\infty} a_i t^i$, while $a_k=c+\sum_{j=0}^{d-n-1} g_j$.
Next we look at (\ref{eq1}) with $t$ substituted by $1-t$:
$$
\frac{Q(1-t)}{t^d}=\sum_{i=0}^{k-1} \frac{f_i (1-t)^i}{t^n} +\frac{c(1-t)^k}{t^n}+ \sum_{j=0}^{d-n-1} \frac{g_j(1-t)^k}{t^{d-j}}.
$$
This time we multiply by $t^d/(1-t)^k$ and get
$$
\frac{\tilde{Q}(t)}{(1-t)^k}=\frac{Q(1-t)}{(1-t)^k}=\sum_{i=0}^{k-1} \frac{f_i t^{d-n}}{(1-t)^{k-i}} +ct^{d-n}+ \sum_{j=0}^{d-n-1} g_jt^{j}, \label{eq2}
$$
hence $g_j=b_j$ for $j=0,\ldots , d-n-1$ and $c=b_{d-n}-\sum_{i=0}^{k-1}f_i$.
\end{proof}

Since the coefficients in the power series expansion of a rational function are unique, the previous lemma has an immediate consequence:

\begin{corollary}
The coefficients in an $(n,k)$-boundary presentation of $H(t)={Q(t)/(1-t)^d}$ are uniquely determined.
\end{corollary}

In the rest of the section we will make extensive use of the relation
\begin{equation}\label{eq:main}
\frac{t^i}{(1-t)^{j}}=\frac{t^{i+1}}{(1-t)^{j}}+\frac{t^i}{(1-t)^{j-1}},~j>1
\end{equation}
\vspace{0.35cm}

\begin{center}
\begin{tikzpicture}[scale=1.2]
    \draw[] (0,0) grid [step=1cm](2,2);
    \node [below right] at (0.27,-0.2) {$\scriptstyle i$};
    \draw[dotted,thick,->] (0.7,1.5)--(1.2,1.5);
    \draw[dotted,thick,->] (0.5,1.3)--(0.5,0.8);
    \node [below right] at (1.1,-0.2) {$\scriptstyle i+1$};          
    \node [below right] at (-0.9,0.75) {$\scriptstyle j-1$};
    \node [below right] at (-0.9,1.75) {$\scriptstyle j$};         
    \node [below right] at (0.3,1.77) {$1$};
    \node [below right] at (1.25,1.7) {$\alpha$};
    \node [below right] at (0.25,0.75) {$\beta$};
    \draw[->] (2.5,1)--(3.2,1);       
\end{tikzpicture}
   \hspace{0.5cm}
\begin{tikzpicture}[scale=1.2]
    \draw[] (0,0) grid [step=1cm](2,2);
    \node [below right] at (0.27,-0.2) {$\scriptstyle i$};
    \node [below right] at (1.1,-0.2) {$\scriptstyle i+1$};          
    \node [below right] at (-0.9,0.75) {$\scriptstyle j-1$};
    \node [below right] at (-0.9,1.75) {$\scriptstyle j$};         
    \node [below right] at (0.3,1.75) {$0$};
    \node [below right] at (1,1.75) {$\alpha+1$};
    \node [below right] at (0,0.8) {$\beta+1$};
   \end{tikzpicture}
\end{center}

Repeated application of this relation allows to transform an $(n,k)$-boundary presentation of a rational function $H$ into an $(n-1,k)$ resp. $(n,k+1)$-boundary presentation.  We give a formula for the coefficients of the new boundary presentation in terms of the old coefficients:

\begin{lemma}\label{lemma:34}
Let
$$
H(t)=\sum_{i=0}^{k-1} \frac{f_i t^i}{(1-t)^n} + \frac{c t^k}{(1-t)^{n}} + \sum_{j=0}^{d-n-1} \frac{g_jt^k}{(1-t)^{d-j}} 
$$
be an $(n,k)$-boundary presentation. Then there exists a corner-free $(n,k+1)$-boundary presentation; its coefficients $f^{(k+1)}, g^{(k+1)}$ are given by
\begin{align*}
f_i^{(k+1)}&=
\left \{ \begin{array}{ll}
f_i & \mbox{for}\ \ \ i=0,\ldots , k-1\\
\\
c+\sum_{r=0}^{d-n-1} g_r &\mbox{for}\ \ \ i=k
\end{array}
\right.\\
g_{j}^{(k+1)}&=\sum_{r=0}^{j} g_{r}, \ \ \mbox{for}\ \ \ j=0,\ldots , d-n-1.
\end{align*} 
If $n>0$ then there is also a corner-free $(n-1,k)$-boundary presentation with
coefficients $f^{(n-1)}, g^{(n-1)}$  given by
\begin{align*}
f_i^{(n-1)}&=\sum_{r=0}^{i} f_r, \ \ \mbox{for}\ \ \ i=0,\ldots , k-1\\
g_j^{(n-1)}&=
\left \{ \begin{array}{ll}
g_j & \mbox{for}\ \ \ j =0, \ldots, d-n-1\\
\\
c+\sum_{r=0}^{k-1} f_r &\mbox{for}\ \ \ j=d-n.
\end{array}
\right.
\end{align*} 
In particular, an expansion of a corner-free boundary presentation leads to a boundary presentation with the entries next to the corner being equal.
\end{lemma}

\begin{corollary} \label{corollary:37a}
Let
$$
H(t)=\sum_{i=0}^{k-1} \frac{f_i t^i}{(1-t)^n} + \sum_{j=0}^{d-n-1} \frac{g_jt^k}{(1-t)^{d-j}} 
$$
be a corner-free $(n,k)$-boundary presentation. If $k >0$ then there exists 
$(n,k-1)$-boundary presentation; its coefficients $f^{(k-1)}, c^{(k-1)}, g^{(k-1)}$ are given by
\begin{align*}
f_i^{(k-1)}&=f_i, \ \ \mbox{for}\ \ \ i=0,\ldots , k-2\\
c^{(k-1)}&=f_{k-1}-g_{d-n-1}\\
\\
g_j^{(k-1)}&=
\left \{ \begin{array}{ll}
g_0 & \mbox{for}\ \ \ j =0\\
g_j- g_{j-1} &\mbox{for}\ \ \ j =1, \ldots , d-n-1.
\end{array}
\right.
\end{align*} 
If $n<d$ then there is also a $(n+1,k)$-boundary presentation with
coefficients $f^{(n+1)}$, $c^{(n+1)}$, $g^{(n+1)}$  given by
\begin{align*}
f_i^{(n+1)}&=
\left \{ \begin{array}{ll}
f_0 & \ \ \mbox{for}\ \ \  i=0\\
f_i-f_{i-1} &\ \ \mbox{for}\ \ \ i=1,\ldots ,k-1\ \ \ \ \ \ \ \ 
\end{array}
\right.\\
\\
c^{(n+1)}&=g_{d-n-1}-f_{k-1}\\
g_{j}^{(n+1)}&= g_{j}, \ \ \mbox{for}\ \ \ j=0,\ldots , d-n-2.
\end{align*} 
\end{corollary}

\begin{corollary} \label{corollary:37}
If a rational function $H$ admits an $(n,k)$-boundary presentation then there is also an $(n',k')$-boundary presentation for every pair $(n',k')$ with $n' \leq n$, $k' \geq k$; for $(n',k') \neq (n,k)$ this presentation is corner-free. Moreover the coefficients of this $(n',k')$-boundary presentation are nonnegative provided that the same holds for the $(n,k)$-boundary presentation. 
\end{corollary}

In particular there exists an $(n,k)$-boundary presentation of $Q(t)/(1-t)^d$ for every $k \geq \deg Q$ and $n=0, \ldots, d-1$; note that in these cases the formula of Lemma \ref{lemma:34} provides an alternative proof for the equality of the coefficients $f_i$ and the first coefficients of ${Q(t)/(1-t)^{d-n}}$.
Analogously, if $d\geq \deg Q$ the $(d-\deg Q,0)$-boundary presentation can be expanded to an $(n,k)$-boundary presentation for $n=0,\ldots , d-\deg Q$ and $k \geq 1$, also confirming the description of the $g_j$.

\begin{corollary} \label{corollary:37b}
If an $(n,k)$-boundary presentation is not corner-free, then it cannot be obtained by expanding some $(n',k')$-boundary presentation with $n' \geq n$, $k' \leq k$. 
\end{corollary}

Since any $(n,k)$-boundary presentation with $k > \deg Q$ can be obtained as an expansion of the $(d,\deg Q$)-boundary presentation of $Q(t)/(1-t)^d$, we get
a second description of the coefficients $g_j$:

\begin{proposition}\label{prop:38}
Let
$$
H(t)=\frac{Q(t)}{(1-t)^d}=\sum_{i=0}^{k-1} \frac{f_it^{i}}{(1-t)^n} + \sum_{j=0}^{d-n-1} \frac{g_j^{(k)}t^k}{(1-t)^{d-j}}
$$
with $k>d$. Then the coefficient $g_{j}^{(k)}$ for $j=1,\ldots , d-n-1$ agrees with the $(k-1)$-th coefficient of the power series expansion of $Q(t)/(1-t)^{j+1}$.

In particular for $Q(t)/(1-t)^k=\sum_{n \geq 0} a^{(k)}_n t^n$ and $\tilde{Q}(t)/(1-t)^{k}=\sum_{n \geq 0} b_n^{(k)} t^n$ we have
$$
b_j^{(k)}=a_{k-1}^{(j+1)} \ \mbox{~for~} k \geq \deg Q \ \mbox{~and~}\ j=0,\ldots, d-1.
$$
\end{proposition}

\begin{proof}
Let $0 \leq j \leq d-1$. We consider the $(d-1-j,k)$-boundary presentation of $H$ with $k>\deg Q$. Since this can be viewed as an expansion of the corner-free $(d,\deg(Q)+1)$-boundary presentation 
$$
\frac{Q(t)}{(1-t)^d}+\frac{0\cdot t^{\deg(Q) +1}}{(1-t)^d}
$$
we have
$f_{k-1}^{(d-1-j)}=g_j^{(k)}$, so by Lemma \ref{lemma:32} $g_j^{(k)}$ agrees with the  $(k-1)$-th coefficient of 
$$
\frac{Q(t)}{(1-t)^{d-(d-1-j)}}=\frac{Q(t)}{(1-t)^{j+1}}.
$$
Expanding the $(d-1-j,k)$-boundary presentation downwards does not affect $g_j^{(k)}$, therefore this equality is also valid for any $(n,k)$-boundary presentation with $n\leq d-1-j$. The second part follows immediately from Lemma \ref{lemma:32}.
\end{proof}

\section{Arithmetical characterization of the Hilbert regularity} \label{sect:ari}

In this section we continue our investigation of the Hilbert regularity, so we restrict our attention to nonnegative series $Q(t)/(1-t)^d$.
As mentioned above, such a series admits a Hilbert decomposition; it is easy to see that it also admits a
boundary presentation with nonnegative coefficients. In the sequel such a boundary presentation will be called \emph{nonnegative} for short.

\begin{lemma}\label{lemma:33}
Let $H(t)=\displaystyle \sum_{i=n}^{d}\frac{Q_i(t)}{(1-t)^i}$ be a Hilbert decomposition, and let $k=\max_{i} \deg{Q_i}$. Then there exists a nonnegative $(n,k)$-boundary presentation of $H$. 
\end{lemma}

\begin{proof}
Obviously a Hilbert decomposition can be rewritten as
\begin{equation}\label{eq4}
\sum_{i=n}^{d} \frac{Q_i(t)}{(1-t)^{i}} = \sum_{j=n}^{d} \sum_{i=0}^{k} \frac{a_{ij}t^i}{(1-t)^j} \mbox{~with~} a_{ij} \in \NN.
\end{equation}
It is enough to show that this decomposition can be turned into one of the form
$$
\sum_{j=n}^p\sum_{i=0}^k \frac{b_{ij}t^i}{(1-t)^j}+\sum_{j=p+1}^{d} \frac{b_{kj}t^k}{(1-t)^j} \mbox{~with~} b_{ij}>0
$$
for any $p$ with $n \leq p \leq d$. Repeated application of the relation (\ref{eq:main}) yields
$$
\sum_{i=0}^k \frac{b_{ij}t^i}{(1-t)^j} = \sum_{i=0}^{k-1} \frac{(\sum_{r=0}^{i}b_{rj})t^i}{(1-t)^{j+1}} + \frac{(\sum_{r=0}^{k}b_{rj})t^k}{(1-t)^j}.
$$

Since the coefficients on the right-hand side are still nonnegative, the claim follows by reverse induction on $p\leq d$, starting with the vacuous case $p=d$.
\end{proof}

\begin{corollary}\label{corollary:42}
(a) Let $H(t)=Q(t)/(1-t)^d$ be a nonnegative series. Then $H$ admits a nonnegative $(0,\Hreg H)$-boundary presentation as well as a nonnegative $(\hdp H,k)$-boundary presentation with suitable $k\geq 0$.

\noindent (b) If $H$ admits a non-corner-free $(0,k)$-boundary presentation, then $\Hreg H \geq k$.
\end{corollary}

\begin{proof}
The statement (a) is clear from the definition of $\Hreg{H}$~resp.~$\hdp{H}$. For (b) assume on the contrary $\Hreg H < k$, then $H$ admits a $(0,\Hreg H)$-boundary presentation, and this presentation could be expanded to the $(0,k)$-boundary presentation, contradicting Corollary \ref{corollary:37b}. 
\end{proof}

\begin{remark}
It is easily seen that, using relation (\ref{eq:main}), an $(n,k)$-boundary presentation with $n,k>0$ can be transformed into a non-corner-free $(n-1,k-1)$-boundary presentation. Hence if 
$\deg{Q}>d$ the rational function $H$ admits a non-corner-free $(0,\deg Q -d)$-boundary presentation; together with part (b) of the corollary this yields another proof of Proposition \ref{lower}.
\end{remark}

Corollary \ref{corollary:42} implies that, for computations of Hilbert regularity (and also of Hilbert depth), we may exclusively consider boundary presentations.
This observation leads to an estimate for $\Hreg M$ in the flavour of the equality $\pos(M)=\hdp M$.
In order to formulate this inequality we need the following notion:

\begin{definition}\label{def:34}
For any $Q\in\ZZ[t]$ and $k \in \NN$, let $Q(t)/(1-t)^k=\sum_{n \geq 0} a_n^{(k)}t^n$.  For any $d \in \NN$ we set
$$
\delta_d (Q):=\min \Big \{ k \in \NN \mid a_0^{(k)}, \ldots , a_{d-1}^{(k)} \mbox{~nonnegative}~ \Big \}
$$
and
$$
\delta(Q):=\min \Big \{ k \in \NN \mid \frac{Q(t)}{(1-t)^k} \mbox{~nonnegative}~ \Big \}.
$$
\end{definition}

Note that $\delta_d (Q)$ is finite if and only if the lowest nonvanishing coefficient of $Q$ is nonnegative, as one sees easily by induction on $d$.
By Theorem 4.7 in \cite{U}, $\delta (Q)$ is finite if and only if $Q$ viewed as a real-valued function of one variable takes positive values in the open interval $(0,1)$.

For a finitely generated graded $R$-module $M$ with Hilbert series $\displaystyle H_M(t)=\frac{Q_M(t)}{(1-t)^{\dim M}}$ the equality $\hdp M=\pos(M)$ implies 
$\delta (Q_M)=\dim M- \hdp M$, so according to Proposition 1.5.15 of \cite{BH} and the Auslander-Buchsbaum theorem, $\delta (Q_M)$ could be named $\Hproj M$, the \emph{Hilbert projective dimension}. Note that $\Hproj M$ only depends on $Q_M$ but not on $\dim M$.
\medskip

The announced estimate for the Hilbert regularity reads as follows:

\begin{proposition}\label{prop:35}
Let $H(t)=Q(t)/(1-t)^d$ be a nonnegative series, then
$$
\Hreg H \geq \delta_d (\tilde{Q}).
$$
\end{proposition}

\begin{proof}
Since $\tilde{Q}(0)=Q(1)>0$, $\delta_d(\tilde{Q})$ is finite.
Let $\Hreg H=k$, then there exists a $(0,k)$-boundary presentation
$$
H(t)=\sum_{i=0}^{k-1} f_i t^i +ct^k +\sum_{j=0}^{d-1} \frac{g_jt^k}{(1-t)^{d-j}}
$$
with nonnegative coefficients. By Lemma \ref{lemma:32} the first $d$ coefficients of
$\tilde{Q}(t)/(1-t)^k$
agree with the coefficients $g_j$ and so they are nonnegative, hence $\delta_d (\tilde{Q}) \leq k =\Hreg H$.
\end{proof}

\begin{proposition}\label{prop:neu}
Under the hypothesis of Proposition \ref{prop:35} we even have
$\Hreg H \geq \delta (\tilde{Q})$.
\end{proposition}

\begin{proof}
An $(n,k)$-boundary presentation of $Q(t)/(1-t)^d$ induces 
an $(n+m,k)$-boundary presentation of $Q(t)/(1-t)^{d+m}$, $m \in \NN$, with the same coefficients. The $(0,\Hreg H)$-boundary presentation of $Q(t)/(1-t)^d$ has nonnegative coefficients, hence the same holds for the $(m,\Hreg H)$-boundary presentation of $Q(t)/(1-t)^{d+m}$, and by Corollary \ref{corollary:37} also the $(0,\Hreg H)$-boundary presentation of $Q(t)/(1-t)^{d+m}$ is nonnegative. This implies $\delta_{d+m} (\tilde{Q}) \leq \Hreg H$ for all $m\in\NN$, and so $\delta (\tilde{Q}) \leq\Hreg H$, as desired.
\end{proof}

\begin{theorem} \label{theorem:partial}
Under the hypothesis of Proposition \ref{prop:35} and the additional assumption of either (i) $\delta_d (\tilde{Q})\geq \deg Q$ or (ii) $\deg Q \leq d$ we have
$$
\Hreg H = \delta_d (\tilde{Q})=\delta (\tilde{Q}).
$$
\end{theorem}

\begin{proof}
In both cases expansion of the $(d,\deg Q)$~resp.~the $(d-\deg Q,0)$-boundary presentation yields a $(0,\delta_d (\tilde{Q}))$-boundary presentation of $H$, which is nonnegative by the nonnegativity of $H$ and the definition of $\delta_d(\tilde{Q})$, and hence 
\begin{equation*}
\delta_d (\tilde{Q}) \geq \Hreg H \geq \delta (\tilde{Q})\geq \delta_d (\tilde{Q}). \tag*{\qedhere}
\end{equation*}
\end{proof}

The following example shows that, contrary to $\hdp M \leq \pos (M)$ in case of the Hilbert depth, the inequality $\Hreg H \geq \delta_d (\tilde{Q})$ may be strict.
 
\begin{example}\label{ex:45}
For $\displaystyle H(t)=\frac{1-t+2t^2-2t^3+t^4}{(1-t)^2}$ we obtain $\tilde{Q}(t)=Q(t)$ 
and therefore 
$$
\frac{\tilde{Q}(t)}{1-t}= \frac{Q(t)}{1-t}=1+0t+2t^2+0t^3+\sum_{n \geq 4} t^n
$$
implies $\delta_2 (\tilde{Q})=1=\Hproj H$. The $(0,2)$-boundary presentation of $H$ is given by
$$
H(t)=1+t+t^2+\frac{t^2}{1-t} + \frac{t^2}{(1-t)^2}.
$$
Since this is not corner-free, Corollary \ref{corollary:37b} implies $\Hreg H =2>1=\delta_2 (\tilde{Q})$. In particular the Hilbert regularity of $Q(t)/(1-t)^d$ 
depends on $d$: For $\displaystyle H'(t)=\frac{1-t+2t^2-2t^3+t^4}{(1-t)^d}$ with $d \geq 4$ we have $\Hreg H' =1$ by Theorem \ref{theorem:partial}.
\end{example}

This example also explains why non-negativity of $\tilde{Q}(t)/(1-t)^k$ for some $k \in \NN$ does not ensure $\Hreg H \leq k$: 
The decomposition
$$
\frac{\tilde{Q}(t)}{(1-t)^k}=\sum_{i=0}^{k} \frac{\tilde{Q}_i(t)}{(1-t)^i}
$$
with nonnegative $\tilde{Q}_i \in \ZZ[t]$ according to Theorem 2.1 in \cite{U} can be turned into one of 
$$\frac{Q(t)}{(1-t)^{\max \{\deg \tilde{Q}_i \}}}$$
by exchanging $t$ and $1-t$, but if $d < \max \{\deg \tilde{Q}_i \}$ this does not yield a decomposition of $Q(t)/(1-t)^d$.
\medskip

Due to the difficulty illustrated by the previous example the general description of the Hilbert regularity is less straightforward than that of the Hilbert depth. 
In the remaining case of $\deg Q > d, \delta (\tilde{Q})$, the $(0,\deg Q)$-boundary presentation is nonnegative and hence $\Hreg H \leq \deg Q$. If $\Hreg H < \deg Q$ then the $(0,\deg Q)$-boundary presentation can be reduced to a nonnegative $(0,k)$-boundary presentation with smaller $k$. Such a reduction could be performed in steps, therefore we investigate whether a reduction from $k$ to $k-1$ is possible:

\begin{proposition} \label{prop:quadrat}
Let 
$$
H(t)=\sum_{i=0}^{k-1} \frac{f_i t^i}{(1-t)^n} + \frac{ct^k}{(1-t)^n}+\sum_{j=0}^{d-n-1} \frac{g_jt^k}{(1-t)^{d-j}} 
$$
with nonnegative coefficients. Then 
$$
\Hreg H \leq k-1 \Longleftrightarrow 
\left \{
\begin{array}{ccl}
c &=&0\\
f_{k-1}&\geq & g_{d-n-1}\\
g_{j+1}&\geq& g_{j} \ \ \mbox{~for~} j=0,\ldots , d-n-2
\end{array}
\right .
$$
\end{proposition} 

\begin{proof} ``$\Longrightarrow$'' \ Let $\Hreg H \leq k-1$, then there exists a boundary presentation
\begin{equation}\label{eq:prime}
H(t)=\sum_{i=0}^{k-2} \frac{f'_i t^i}{(1-t)^n} + \frac{c't^{k-1}}{(1-t)^n}+\sum_{j=0}^{d-n-1} \frac{g'_jt^{k-1}}{(1-t)^{d-j}} 
\end{equation}
with nonnegative coefficients. By Lemma \ref{lemma:34}, this presentation can be transformed into 
$$
H(t)=\sum_{i=0}^{k-2} \frac{f_i t^i}{(1-t)^n} +\frac{(c'+\sum_{j=0}^{d-n-1} g'_j) t^{k-1}}{(1-t)^n}+ \sum_{j=0}^{d-n-1} \frac{(\sum_{i=0}^{j}g'_i) t^k}{(1-t)^{d-j}},
$$
and by uniqueness of the $(n,k)$-boundary presentation we have
$$
f_{k-1}= c'+\sum_{j=0}^{d-n-1} g'_j \geq \sum_{j=0}^{d-n-1} g'_j =g_{d-n-1}.
$$
The necessity of the other conditions was already noted in Corollary \ref{corollary:42}~(b) and Proposition \ref{prop:35}.
\medskip

``$\Longleftarrow$''\ If the conditions on the right are satisfied then Corollary \ref{corollary:37a} yields a nonnegative $(0,k-1)$-boundary presentation (\ref{eq:prime}).
\end{proof}

The $(0,\Hreg H)$-boundary presentation can be achieved by iterated reduction steps starting from the $(0,\deg Q)$-boundary presentation.
The reduction continues as long as the conditions of the previous proposition remain valid. Hence it stops in one of the three cases illustrated by the following diagrams

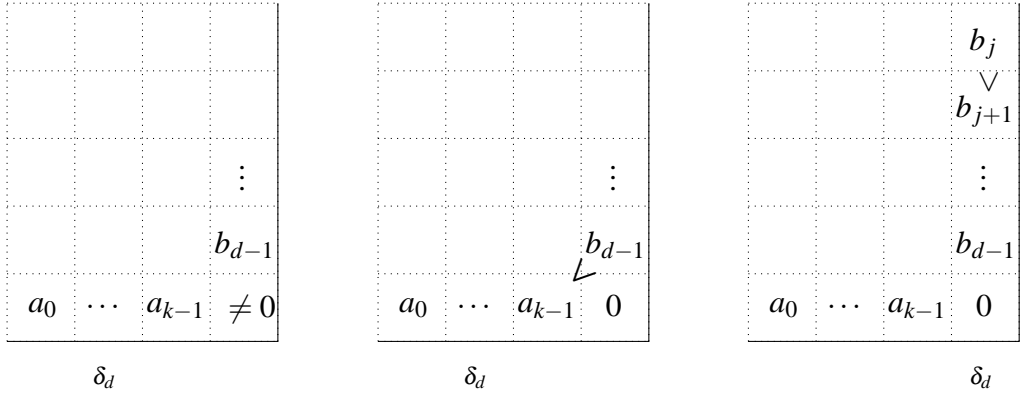
\begin{center}
\begin{tikzpicture}[scale=0.9]
    \draw[dotted] (0,0) grid [step=1cm](4,5);
    \draw[-] (0,0)--(4,0)--(4,5);
    \node [below right] at (1.1,-0.2) {$\scriptstyle \delta_d$};
    \node [below right] at (0.15,0.8) {$a_0$};
    \node [below right] at (1,0.7) {$\ldots$};
    \node [below right] at (1.9,0.8) {$a_{k-1}$};
    \node [below right] at (3.1,0.85) {$\neq 0$};
    \node [below right] at (2.9,1.8) {$b_{d-1}$};
    \node [below right] at (3.25,3) {$\vdots$};
\end{tikzpicture}
  \hspace{1cm}
\begin{tikzpicture}[scale=0.9]
    \draw[dotted] (0,0) grid [step=1cm](4,5);
    \draw[-] (0,0)--(4,0)--(4,5);
    \node [below right] at (1.1,-0.2) {$\scriptstyle \delta_d$};
    \node [below right] at (0.15,0.8) {$a_0$};
    \node [below right] at (1,0.7) {$\ldots$};
    \node [below right] at (1.9,0.8) {$a_{k-1}$};
    \node [below right] at (3.2,0.85) {$0$}; 
    \node [below right] at (2.9,1.8) {$b_{d-1}$};
    \node [below right] at (3.25,3) {$\vdots$};
    \node [ultra thick,rotate=45] at (3,1) {$\bm{<}$};
\end{tikzpicture}
 \hspace{1cm}
\begin{tikzpicture}[scale=0.9]
    \draw[dotted] (0,0) grid [step=1cm](4,5);
    \draw[-] (0,0)--(4,0)--(4,5);

    \node [below right] at (3.1,-0.2) {$\scriptstyle \delta_d$};
    \node [below right] at (0.15,0.8) {$a_0$};
    \node [below right] at (1,0.7) {$\ldots$};
    \node [below right] at (1.9,0.8) {$a_{k-1}$};
    \node [below right] at (3.2,0.85) {$0$};
    \node [below right] at (2.9,1.8) {$b_{d-1}$};
    \node [rotate=90,below right] at (3.25,3.5) {$<$};
    \node [below right] at (3.25,3) {$\vdots$};
    \node [below right] at (2.9,3.8) {$b_{j+1}$};
    \node [below right] at (3.1,4.8) {$b_{j}$};
\end{tikzpicture}

\end{center}

The construction of the $(0, \Hreg H)$-boundary presentation can be described as follows: Starting with $k=\deg Q$
we consider the $(0,k)$-boundary presentation. As long as ${k>\delta_d(\tilde{Q})}$ and $f_{k-1}=g_{d-n-1}^{(k)} $ there is also a nonnegative and corner-free $(0,k-1)$-boundary presentation, so we continue with $k-1$ instead of $k$. As soon as $k=\delta_d (\tilde{Q})$ or $f_{k-1}\neq g_{d-n-1}^{(k)}$ we have reached the minimal $k$ for which a nonnegative and corner-free $(0,k)$-boundary presentation exists. If $k=\delta_d (\tilde{Q})$ or $f_{k-1}< g_{d-n-1}^{(k)}$ no further reduction is possible, hence
$\Hreg H=k$, but if $k>\delta_d(\tilde{Q})$ and $f_{k-1} \geq g_{d-n-1}^{(k)}$ one last reduction step, leading to a non-corner-free boundary presentation, can be performed, so $\Hreg H=k-1$ in this case.

\begin{theorem} \label{theorem:complete}
Let $H(t)=Q(t)/(1-t)^d=\sum_{n \geq 0} a_n t^n$ be a nonnegative series with $d>0$, and let $\tilde{Q}(t)/(1-t)^{j}=\sum_{n \geq 0} b_n^{(j)} t^n$ for $j \in \mathbb{N}$. 
\begin{itemize}
\item[(i)] If $\deg Q \leq d$ or $\delta_d (\tilde{Q}) \geq \deg Q$, then $\Hreg H =\delta_d (\tilde{Q})$.
\item[(ii)] Otherwise, with
$$
k:=\min \{ i \mid  \delta_d (\tilde{Q}) \leq i \leq  \deg Q \  \ \mbox{and} \ \  a_{j}=b_{d-1}^{(j+1)} \ \ \mbox{for~all~} j=i, \ldots , \deg Q\}
$$
we have
$$
\Hreg H=\left \{
\begin{array}{ll}
k&\mbox{if~} k=\delta_d(\tilde{Q}) ~\vee~  a_{k-1}<b_{d-1}^{(k)}\\
k-1&\mbox{if~} k>\delta_d(\tilde{Q}) ~\wedge~  a_{k-1}>b_{d-1}^{(k)}.
\end{array}
\right.
$$
\end{itemize}
\end{theorem}

\begin{proof}
The cases in (i) were already treated in Theorem \ref{theorem:partial}. 
Part (ii) follows from the discussion preceding this theorem; the number $k$, which is well-defined by Proposition \ref{prop:38}, is just the width of the minimal nonnegative and corner-free boundary presentation.
\end{proof}

The closing result of this section is the analogue of Proposition \ref{prop:neu} for $\delta (Q)$. 

\begin{lemma}
Let $H(t)=Q(t)/(1-t)^d$ be nonnegative and $e:={\max\{\delta_d(\tilde{Q}), \deg{(Q)}+1 \}}$. Then $\delta (Q)=\delta_e(Q)$
\end{lemma}

\begin{proof}
The $(d-\delta_e(Q),\delta_e(Q))$-boundary presentation of $H$ is nonnegative by Lemma \ref{lemma:32} and the definition of $\delta_d(\tilde{Q})$ and $\delta_e(Q)$. Hence the $(d-\delta_e(Q),\delta_{e+m}(Q))$-boundary presentation with $m\geq 0$ is nonnegative as well, but this implies $\delta_{e+m}(Q)\leq \delta_e(Q)$ for all $m \in \NN$, therefore $\delta (Q)=\delta_e(Q)$.
\end{proof}

\section{Computation of Hilbert depth and Hilbert regularity}\label{sect:alg}

The aim of this section is an algorithm for computing the Hilbert depth and Hilbert regularity of a module with given Hilbert series $H(t)=Q(t)/(1-t)^d$. An algorithm solely for the Hilbert depth was given by A. Popescu in \cite{P}.
\medskip

\begin{algorithm}[H]
\caption{Computing Hilbert depth and Hilbert regularity}\label{algorithm}
\SetAlgoLined
\SetKwInOut{Input}{Input}
\SetKwInOut{Output}{Output}
\textbf{Input}:~$Q \in \ZZ[t], d \in \ZZ$ with $H(t)=Q(t)/(1-t)^d$ nonnegative \\
\nl $\tilde{Q}(t):=Q(1-t)$\;
\BlankLine
\nl - - Determine $\delta_d (\tilde{Q})$: 

$k:=-1$\;
\Repeat{
$b_{0}^{(k)}, \ldots , b_{d-1}^{(k)}$ nonnegative
}{
$k:=k+1$\;
Compute the first $d$ coefficients $b_{0}^{(k)}, \ldots ,b_{d-1}^{(k)}$ of $\tilde{Q}(t)/(1-t)^{k}$\;
}
$\delta_d(\tilde{Q})=k$\;

\BlankLine
\nl - - Determine $\Hproj H$: 

$e:=\max \{\delta_d (\tilde{Q}),\deg (Q) +1\}$\;
$k:=-1$\;
\Repeat{
$a_{0}^{(k)}, \ldots , a_{e-1}^{(k)}$ nonnegative
}{
$k:=k+1$\;
Compute the first $e$ coefficients $a_{0}^{(k)}, \ldots ,a_{e-1}^{(k)}$ of $Q(t)/(1-t)^{k}$\;
}
$\Hproj H=k$\;
\BlankLine
\nl $\Hdepth H=d-\Hproj H$\;
\BlankLine
\nl - - Determine $\Hreg H$:

\eIf{$\deg Q \leq d$ {\bf or} $\delta_d (\tilde{Q}) \geq \deg Q$}{
$\Hreg H=\delta_d (\tilde{Q})$\;
}{
Compute the $i$-th coefficient $a_i$ of $H$ for $i=0, \ldots , \deg Q$\;
Compute the $(d-1)$-th coefficient $b_{d-1}^{(j)}$ of $\displaystyle \frac{\tilde{Q}(t)}{(1-t)^j}$ for $j=\delta_d (\tilde{Q}),  \ldots , \deg Q$\;
$k:=\min \{ i \mid  \delta_d (\tilde{Q}) \leq i \leq  \deg Q   ~\mbox{and}~  a_{j}=b_{d-1}^{(j+1)} ~\mbox{for~all~} j=i, \ldots , \deg Q\}$\;
\eIf{$a_{k-1} \geq b_{d-1}^{(k)}$ {\bf and} $k>\delta_d (\tilde{Q})$}{
$\Hreg H=k-1$\;
}{
$\Hreg H=k$\;
}
}
\Output{$\Hdepth H, \Hreg H$}
\end{algorithm}
\medskip

The correctness of this algorithm follows immediately from the previous results.
The output could be easily extended by the boundary presentations realising $\Hdepth$ or $\Hreg$,
since the required coefficients are computed in the course; for example, a nonnegative boundary presentation 
of the minimal height $\Hdepth H$ is given by
$$
H(t)=\sum_{i=0}^{e-1} \frac{a_{i}^{(h)} t^i}{(1-t)^h} + \left \{
\begin{array}{ll}
\displaystyle \sum_{j=0}^{d-h-1} \frac{a_{e-1}^{(j+1)}t^e}{(1-t)^{d-j}} & \mbox{for~} e=\deg Q >\delta_d (\tilde{Q})\\
&\\
\displaystyle \sum_{j=0}^{d-h-1} \frac{b_{j}^{(\delta_d (\tilde{Q}))}t^e}{(1-t)^{d-j}} & \mbox{for~} e=\delta_d (\tilde{Q}) \geq \deg Q
\end{array}
\right.
$$
with $a$ and $b$ used as in the description of the algorithm, and $h:=\Hproj H$.
\medskip

For completeness we give an upper bound for the number of repetitions of the loop in the second step of Algorithm \ref{algorithm}. The idea is to replace $\tilde{Q}(t)=\sum_i \tilde{q}_i t^i$ with 
a polynomial $\tilde{q}_0+rt$ such that for all $n,i \in \NN$ the coefficient $c_n^{(k)}$ of $(\tilde{q}_0+rt)/(1-t)^k$ is not greater than the coefficient $b_n^{(k)}$ 
of $\tilde{Q}(t)/(1-t)^k$. Such a polynomial can be obtained by repeated application of the map
$$
f=\sum_{i=0}^{m} h_i t^i \longmapsto \sum_{i=0}^{m-2} h_i t^i +\min \{h_{m-1},h_{m-1}+h_m \}t^{m-1}
$$
to the polynomial $\tilde{Q}$. Since
\begin{align*}
\frac{\tilde{q}_0+rt}{(1-t)^k} &=\sum_{n \geq 0} \left [ \tilde{q}_0 {n+k-1 \choose n-1} + r {n+k-2 \choose n-2}   \right ] t^n\\
&= \sum_{n \geq 0} \left [ \frac{\prod_{j=0}^{k-2} (n+j)}{k !} \left (\tilde{q}_0 (n+k-1)+r(n-1) \right) \right ] t^n,
\end{align*}
we want to determine the least $k$ such that
\begin{equation}\label{eq:tilde}
\tilde{q}_0 (n+k-1)+r(n-1)  = (\tilde{q}_0+r)(n-1)+k\tilde{q}_0 \geq 0
\end{equation}
holds for $0 \leq n \leq d-1$. Without loss of generality we may assume $q+r<0$. Then (\ref{eq:tilde}) is equivalent to 
$$
n \leq 1-\frac{\tilde{q}_0 k}{\tilde{q}_0+r}.
$$
This inequality has to be valid in particular for $n=d-1$, and so for
$$
k \geq \frac{(2-d)(\tilde{q}_0 +r)}{\tilde{q}_0}
$$
the first $d$ coefficients of $(\tilde{q}_0+rt)/(1-t)^k$ and a fortiori those of $\tilde{Q}(t)/(1-t)^k$ are nonnegative.

\begin{example}
Let $\displaystyle H(t)=\frac{2-5t+t^2+4t^3}{(1-t)^7}$. Then 
$\tilde{Q}(t)=Q(1-t)=2-9t+13t^2-4t^3$ and we find $\delta_7(\tilde{Q})=7$ since 
\begin{align*}
\frac{\tilde{Q}(t)}{(1-t)^5}&= 2 +t-2t^2-4t^3+0t^4+17t^5+56t^6+\ldots\\
\frac{\tilde{Q}(t)}{(1-t)^6}&= 2 +3t+t^2-3t^3-3t^4+14t^5+70t^6+\ldots\\
\frac{\tilde{Q}(t)}{(1-t)^7}&= 2 +5t+6t^2+3t^3+0t^4+14t^5+84t^6+\ldots
\end{align*}
In order to determine the Hilbert depth we compute the first $\delta_7 (\tilde{Q})=7$ coefficients of $Q(t)/(1-t)^k$ for $k \geq 0$. Since 
\begin{align*}
\frac{Q(t)}{(1-t)^5}&= 2 +5t+6t^2+4t^3+0t^4-3t^5+0t^6+\ldots\\
\frac{Q(t)}{(1-t)^6}&= 2 +7t+13t^2+17t^3+17t^4+14t^5+14t^6+\ldots
\end{align*}
we have $\Hdepth{H}=7-6=1$.
The Hilbert regularity requires no further computations since
$\deg Q=3 <7=d$, and so $\Hreg{H}=\delta_7 (\tilde{Q})=7$; moreover in this case the boundary presentation
\begin{align*}
H(t)&=\displaystyle\frac{2+7t+13t^2+17t^3+17t^4+14t^5+14t^6}{1-t} \\
&+ \displaystyle\frac{14t^7}{(1-t)^2} +\frac{3t^7}{(1-t)^4} +\frac{6t^7}{(1-t)^5} +  \frac{5t^7}{(1-t)^6} +\frac{2t^7}{(1-t)^7}
\end{align*}
simultaneously has the minimal height $\Hdepth H$ and the minimal width $\Hreg H$.
\end{example}

Finally we give two examples illustrating the case $\Hreg H > \delta_d (\tilde{Q})$.

\begin{example}
Let $\displaystyle H(t)=\frac{1-t+t^3}{(1-t)^2}$. Then $\tilde{Q}(t)=1-2t+3t^2-t^3$ and $\delta_2(\tilde{Q})=2$. 
Since $\deg Q$ exceeds $\delta_d (\tilde{Q})$ as well as $d$, the final loop of our algorithm applies.
By
\begin{align*}
\frac{Q(t)}{(1-t)^2}&=1+t+t^2+2t^3+\ldots\\
\frac{\tilde{Q}(t)}{(1-t)^2}&=1+0t+\ldots\\
\frac{\tilde{Q}(t)}{(1-t)^3}&=1+t+\ldots
\end{align*}
we find $k=2=\delta_2(\tilde{Q})$, hence $\Hreg H=2$.
 \medskip

\begin{center}
\begin{tikzpicture}[scale=0.8]
    \draw[] (0,0) grid [step=1cm](4,3);

    \node [below right] at (0.1,-0.2) {$\scriptstyle 0$};
    \node [below right] at (1.1,-0.2) {$\scriptstyle 1$};
    \node [below right] at (2.1,-0.2) {$\scriptstyle 2$};
    \node [below right] at (3.1,-0.2) {$\scriptstyle 3$};
 
    \node [below right] at (-0.9,0.8) {$\scriptstyle 0$};
    \node [below right] at (-0.9,1.8) {$\scriptstyle 1$};
    \node [below right] at (-0.9,2.8) {$\scriptstyle 2$};
         
    \node [below right] at (0.25,2.8) {$1$};
    \node [below right] at (1,2.8) {$-1$};
    \node [below right] at (2.25,2.8) {$0$};
    \node [below right] at (3.25,2.8) {$1$};
    \draw[->] (4.5,1.5)--(5.3,1.5);      
\end{tikzpicture}
   \hspace{0.1cm}
\begin{tikzpicture}[scale=0.8]
   \draw[] (0,0) grid [step=1cm](4,3);
      \node [below right] at (0.1,-0.2) {$\scriptstyle 0$};
      \node [below right] at (1.1,-0.2) {$\scriptstyle 1$};
      \node [below right] at (2.1,-0.2) {$\scriptstyle 2$};
      \node [below right] at (3.1,-0.2) {$\scriptstyle 3$};

      \node [below right] at (-0.9,0.8) {$\scriptstyle 0$};
      \node [below right] at (-0.9,1.8) {$\scriptstyle 1$};
      \node [below right] at (-0.9,2.8) {$\scriptstyle 2$};
         
      \node [below right] at (0.25,0.85) {$1$};
      \node [below right] at (1.25,0.85) {$1$};
      \node [below right] at (2.25,0.85) {$1$};
      \node [below right] at (3.25,1.85) {$1$};
      \node [below right] at (3.25,2.85) {$1$};
      \node [below right] at (3.25,0.85) {$0$};
      \draw[->] (4.5,1.5)--(5.3,1.5);  
\end{tikzpicture}
    \hspace{0.1cm}
\begin{tikzpicture}[scale=0.8]
   \draw[] (0,0) grid [step=1cm](4,3);
   \node [below right] at (0.1,-0.2) {$\scriptstyle 0$};
   \node [below right] at (1.1,-0.2) {$\scriptstyle 1$};
   \node [below right] at (2.1,-0.2) {$\scriptstyle 2$};
   \node [below right] at (3.1,-0.2) {$\scriptstyle 3$};

   \node [below right] at (-0.9,0.8) {$\scriptstyle 0$};
   \node [below right] at (-0.9,1.8) {$\scriptstyle 1$};
   \node [below right] at (-0.9,2.8) {$\scriptstyle 2$};
         
   \node [below right] at (0.25,0.85) {$1$};
   \node [below right] at (1.25,0.85) {$1$};
   \node [below right] at (2.25,0.85) {$0$};
           
   \node [below right] at (2.25,1.85) {$0$};
   \node [below right] at (2.25,2.85) {$1$};
\end{tikzpicture}

\end{center}
\begin{center}
{\small $H(t)=(1-t+t^3)/(1-t)^2$.}
\end{center}
\medskip

This example confirms that $\Hreg H=\delta_d(\tilde{Q})$ may also occur if $\deg Q >d,\delta_d(\tilde{Q})$.
\end{example}

\begin{example}
For $\displaystyle H(t)=\frac{1-t+2t^2-t^3}{(1-t)^2}$ we have $\delta_2(\tilde{Q})=1$, and the calculations can be summarized by
 \medskip

\begin{center}
\begin{tikzpicture}[scale=0.8]
    \draw[] (0,0) grid [step=1cm](4,3);

      \node [below right] at (0.1,-0.2) {$\scriptstyle 0$};
      \node [below right] at (1.1,-0.2) {$\scriptstyle 1$};
      \node [below right] at (2.1,-0.2) {$\scriptstyle 2$};
      \node [below right] at (3.1,-0.2) {$\scriptstyle 3$};
 
      \node [below right] at (-0.9,0.8) {$\scriptstyle 0$};
      \node [below right] at (-0.9,1.8) {$\scriptstyle 1$};
      \node [below right] at (-0.9,2.8) {$\scriptstyle 2$};
         
      \node [below right] at (0.25,2.8) {$1$};
      \node [below right] at (1,2.8) {$-1$};
      \node [below right] at (2.25,2.8) {$2$};
      \node [below right] at (3,2.8) {$-1$};
      \draw[->] (4.5,1.5)--(5.3,1.5);      
\end{tikzpicture}
   \hspace{0.1cm}
\begin{tikzpicture}[scale=0.8]
   \draw[] (0,0) grid [step=1cm](4,3);
   \node [below right] at (0.1,-0.2) {$\scriptstyle 0$};
   \node [below right] at (1.1,-0.2) {$\scriptstyle 1$};
   \node [below right] at (2.1,-0.2) {$\scriptstyle 2$};
   \node [below right] at (3.1,-0.2) {$\scriptstyle 3$};

   \node [below right] at (-0.9,0.8) {$\scriptstyle 0$};
   \node [below right] at (-0.9,1.8) {$\scriptstyle 1$};
   \node [below right] at (-0.9,2.8) {$\scriptstyle 2$};
         
   \node [below right] at (0.25,0.85) {$1$};
   \node [below right] at (1.25,0.85) {$1$};
   \node [below right] at (2.25,0.85) {$3$};
   \node [below right] at (3.25,1.85) {$3$};
   \node [below right] at (3.25,2.85) {$1$};
   \node [below right] at (3.25,0.85) {$0$};
   \draw[->] (4.5,1.5)--(5.3,1.5);  
\end{tikzpicture}
    \hspace{0.1cm}
\begin{tikzpicture}[scale=0.8]
   \draw[] (0,0) grid [step=1cm](4,3);
   \node [below right] at (0.1,-0.2) {$\scriptstyle 0$};
   \node [below right] at (1.1,-0.2) {$\scriptstyle 1$};
   \node [below right] at (2.1,-0.2) {$\scriptstyle 2$};
   \node [below right] at (3.1,-0.2) {$\scriptstyle 3$};

   \node [below right] at (-0.9,0.8) {$\scriptstyle 0$};
   \node [below right] at (-0.9,1.8) {$\scriptstyle 1$};
   \node [below right] at (-0.9,2.8) {$\scriptstyle 2$};
         
   \node [below right] at (0.25,0.85) {$1$};
   \node [below right] at (1.25,0.85) {$1$};
   \node [below right] at (2.25,0.85) {$0$};
           
   \node [below right] at (2.25,1.85) {$2$};
   \node [below right] at (2.25,2.85) {$1$};
   \end{tikzpicture}

\end{center}
\begin{center}
{\small $H(t)=(1-t+2t^2-t^3)/(1-t)^2$.}
\end{center}
\medskip

The third subcase of $\Hreg{H}>\delta_d(\tilde{Q})$, leading to a non-corner-free $(0,\Hreg H)$-boundary presentation, already appeared in
Example \ref{ex:45}.
\end{example}


\begin{thebibliography}{99}

\bibitem{Bi} A.~M.~Bigatti, \emph{Upper bounds for
    the Betti numbers of a given Hilbert function.} Comm. Algebra {\bf
    21} (1993), 2317--2334.

\bibitem{BH} W.~Bruns, J.~Herzog, \emph{Cohen-Macaulay Rings},
    Rev. Ed., Studies in Advanced Mathematics, vol.
    39. Cambridge University Press, 1996.

\bibitem{BKU} W.~Bruns, C.~Krattenthaler, and J.~Uliczka, {\em
    Stanley decompositions and Hilbert depth in the Koszul
    complex.} J.~Commut.\ Algebra {\bf 2} (2010), 327--359.

\bibitem{ERV} J.~Elias, L.~Robbiano, and G.~Valla, \emph{Number
    of generators of ideals.}
    Nagoya Math. J. {\bf 123} (1991), 39--76.

\bibitem{H} J.~Herzog, {\em A survey on Stanley depth.} In
    ``Monomial Ideals, Computations and Applications",
    A.~Bigatti, P.~Gim\'enez, E.~S\'aenz-de-Cabez\'on (Eds.),
    Proceedings of MONICA 2011. Lecture Notes in Math.~2083,
    Springer~(2013).

\bibitem{Hu} H.~A.~Hulett, \emph{Maximum Betti numbers of
    homogeneous ideals with a  given Hilbert function.}
    Comm.~Algebra {\bf 21} (1993), 2335--2350.
    
    \bibitem{IM} B. Ichim, J.~J.~Moyano-Fern\'andez, \emph{How to compute the multigraded Hilbert depth of a module.} 
    Preprint, arXiv: 1209.0084v2 (2012).
    
    \bibitem{MU}
J.~J.~Moyano-Fern\'andez, J.~Uliczka, \emph{Hilbert depth 
of graded modules over polynomial rings in two variables.}
               J.~Algebra \textbf{373} (2013), 130--152.

\bibitem{P} A.~Popescu, {\em An algorithm to compute the
    Hilbert depth.} Preprint, arXiv: 1307.6084v2 (2013)

\bibitem{S} A.~S.~Jahan, {\em Prime filtrations and
    Stanley decompositions of squarefree modules and Alexander
    duality.} Manuscr.~Math. {\bf 130}~(4) (2009), 533--550.

\bibitem{U} J.~Uliczka, {\em Remarks on Hilbert series of
    graded modules over polynomial rings.} Manuscr.~Math. {\bf
    132} (2010), 159--168.

\end{thebibliography}
\end{document}